\documentclass[a4paper]{amsart}

\usepackage{amsmath}
\usepackage{amsthm}
\usepackage{amsfonts}
\usepackage[citation-order,nobysame]{amsrefs}

\usepackage{bbold}
\usepackage{hyperref}
\usepackage{tikz}

\newcommand{\ol}{\overline}
\newcommand{\wh}{\widehat}
\newcommand{\wt}{\widetilde}
\newcommand{\coleq}{\mathrel{\mathop:}=}

\newcommand{\LL}{L(\Omega)}

\newcommand{\mquad}{\mkern-18mu}
\newcommand{\mqquad}{\mkern-36mu}

\theoremstyle{definition}
\newtheorem{definition}{Definition}[section]
\newtheorem{assumption}[definition]{Assumption}

\theoremstyle{plain}
\newtheorem{lemma}[definition]{Lemma}
\newtheorem{proposition}[definition]{Proposition}
\newtheorem{theorem}[definition]{Theorem}
\newtheorem{corollary}[definition]{Corollary}

\begin{document}
\title[Trend to equilibrium for RDD--Poisson models]{Uniform convergence to equilibrium for a family of drift--diffusion models with trap-assisted recombination 
and self-consistent potential}

\author{Klemens Fellner}
\address{Institute of Mathematics and Scientific Computing, University of Graz, Heinrichstra\ss e 36, 8010 Graz, Austria}
\email{klemens.fellner@uni-graz.at}

\author{Michael Kniely}
\address{Faculty of Mathematics, TU Dortmund University, Vogelpothsweg 87, 44227 Dortmund, Germany}
\email[Corresponding author]{michael.kniely@tu-dortmund.de}




\begin{abstract}
We investigate a recombination--drift--diffusion model coupled to Poisson's equation modelling the transport of charge within certain types of semiconductors. In more detail, we study a two-level system for electrons and holes endowed with an intermediate energy level for electrons occupying trapped states. As our main result, we establish an explicit functional inequality between relative entropy and entropy production, which leads to exponential convergence to equilibrium. We stress that our approach is applied uniformly in the lifetime of electrons on the trap level assuming that this lifetime is sufficiently small. 
\end{abstract}

\maketitle


\section{Introduction and main results}

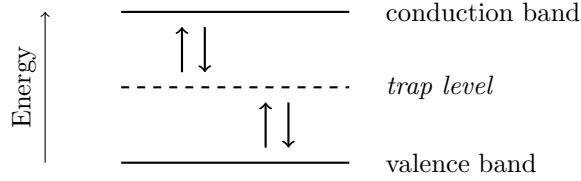
\begin{figure}[t]
\label{figmodel}
\centering
\begin{tikzpicture}
\draw[->] (-1,0) -- node[sloped, anchor=south]{Energy} (-1,2);
\draw[thick] (0,0) -- (3,0) node[anchor=west]{\quad valence band};
\draw[thick, dashed] (0,1) -- (3,1) node[anchor=west]{\quad \textit{trap level}};
\draw[thick] (0,2) -- (3,2) node[anchor=west]{\quad conduction band};
\draw[thick, ->] (.8, 1.2) -- (.8, 1.8);
\draw[thick, <-] (1.1, 1.2) -- (1.1, 1.8);
\draw[thick, ->] (1.9, .2) -- (1.9, .8);
\draw[thick, <-] (2.2, .2) -- (2.2, .8);
\end{tikzpicture}
\caption{An illustration of the allowed transitions of electrons between the three energy levels.}
\end{figure}

We consider the following PDE--ODE recombination--drift--diffusion system coupled to Poisson's equation on a bounded domain $\Omega \subset \mathbb{R}^3$ with boundary $\partial \Omega \in C^2$: 
\begin{equation}
\label{eqsystem}
\begin{cases}
\begin{aligned}
\partial_t n &= \nabla \cdot J_n(n, \psi) + R_n(n,n_{tr}), \\
\partial_t p &= \nabla \cdot J_p(p, \psi) + R_p(p,n_{tr}), \\
\varepsilon \, \partial_t n_{tr} &= R_p(p, n_{tr}) - R_n(n, n_{tr}), \\
- \lambda \, \Delta \psi &= n - p + \varepsilon n_{tr} - D,
\end{aligned}
\end{cases}
\end{equation}
where the flux terms $J_n$, $J_p$ and recombination terms $R_n$, $R_p$ are defined as
\begin{align*}
& & J_n &:= \nabla n + n\nabla (\psi + V_n) = \mu_n \nabla \frac{n}{\mu_n} + n \nabla \psi,  & \mu_n &:= e^{-V_n}, & & \\
& & J_p &:= \nabla p + p\nabla (-\psi + V_p) = \mu_p \nabla \frac{p}{\mu_p} - p \nabla \psi, & \mu_p &:= e^{-V_p}, & &
\end{align*}
\[
R_n := \frac{1}{\tau_n} \left( n_{tr} - \frac{n}{n_0 \mu_n} (1 - n_{tr}) \right), \qquad
R_p := \frac{1}{\tau_p} \left( 1 - n_{tr} - \frac{p}{p_0 \mu_p} n_{tr} \right).
\]

The variables $n$, $p$, and $n_{tr}$ denote the densities of electrons in the conduction band, holes in the valence band, and electrons on the trap level (see Fig.\ \ref{figmodel}). Moreover, $\psi$ represents the electrostatic potential generated by $n$, $p$, $n_{tr}$, and the time-independent doping profile $D \in L^{\infty}(\Omega)$. The constants $n_0, p_0, \tau_n, \tau_p>0$ are positive recombination parameters, while $\varepsilon\in(0,\varepsilon_0]$ (for arbitrary but fixed $\varepsilon_0>0$) is a dimensionless quantity which can be interpreted, on the one hand, as the density of available trapped states and, on the other hand, as the lifetime of electrons on the trap level. 
Note that system \eqref{eqsystem} reduces in the limit $\varepsilon=0$ to the famous Shockley--Read--Hall model of electron recombination \cite{SR52,H52} in semiconductor drift--diffusion systems, see e.g. \cite{MRS90, GMS07}.
Finally,
\begin{align}
\label{eqdefpot}
V_n, \, V_p \in W^{2,\infty}(\Omega) \qquad \mbox{with} \qquad \hat n \cdot \nabla V_n = \hat n \cdot \nabla V_p = 0 \quad \mbox{on} \ \partial \Omega 
\end{align}
represent external time-independent potentials.

The system is equipped with no-flux boundary conditions for the flux components and homogeneous Neumann boundary data for the electrostatic potential,
\begin{equation}
\label{eqboundaryconditions}
\hat n \cdot J_n = \hat n \cdot J_p = \hat n \cdot \nabla \psi = 0 \quad \mbox{on} \ \partial \Omega
\end{equation}
where $\hat n$ represents the outer unit normal vector on $\partial \Omega$. As we will use a result in \cite{GMS07} for proving existence and uniqueness of global solutions, we choose the initial conditions $n(0, \cdot) = n_I$, $p(0, \cdot) = p_I$, and $n_{tr}(0, \cdot) = n_{tr,I}$ in accordance with \cite{GMS07}:
\begin{equation}
\label{eqinitialconditions}
n_I, p_I \in H^1(\Omega) \cap L^\infty(\Omega), \quad n_I, p_I \geq 0, \quad 0 \leq n_{tr,I} \leq 1.
\end{equation}
For convenience, we assume that the volume of $\Omega$ is normalised, i.e. $|\Omega| = 1$, and we set 
\[
\ol{f} := \int_{\Omega} f(x) \, dx
\]
for any function $f \in L^1(\Omega)$. This abbreviation is consistent with the usual definition of the average of $f$ since $|\Omega|=1$.

The main goal of the paper is to close the gap between the models investigated in \cite{FK18} and \cite{FK20}. While the pure Shockley--Read--Hall model including the electrostatic potential has been considered in \cite{FK18}, the family of drift--diffusion models with trap-assisted recombination already appeared in \cite{FK20} but without coupling to Poisson's equation. Here, we focus on the exponential convergence to equilibrium for a PDE--ODE model including both trapped states dynamics and the self-consistent potential. More precisely, we obtain an explicit bound for the convergence rate by employing the so-called \emph{entropy method} which amounts to deriving a functional inequality between an entropy functional and the associated entropy production. For related works on exponential convergence to equilibrium for reaction--diffusion systems, see e.g.\ \cite{AMT00,DFT17_ComplexBalance,FT17_DetailBalance,FT18_ConvergenceOfRenormalizedSolutions,MHM15}. The framework of the entropy method which we shall use here to obtain explicit bounds on the convergence rate originates from \cite{DF06,DF08,DF14}, where models from reversible chemistry have been studied. An earlier application of the entropy method, but using a non-constructive compactness argument, is presented in \cite{GGH96,GH97}, where the authors prove exponential convergence for a model of electrically charged species taking the coupling to Poisson's equation into account.

Throughout the article, we will frequently encounter the following inhomogeneous Poisson equation with right hand side $f \in L^2(\Omega)$ subject to homogeneous Neumann boundary conditions: 
\begin{equation}
\label{eqpsi}
-\lambda \, \Delta\psi = f \quad \mbox{in} \ \Omega, \qquad
\hat n \cdot \nabla \psi = 0 \quad \mbox{on} \ \partial \Omega.
\end{equation}
It is well-known that there exists a weak solution $\psi \in H^1(\Omega)$ if and only if $\ol f = 0$ holds true (compatibility condition with homogeneous Neumann boundary data). 
In this case, $\psi$ is determined only up to an additive constant, which one can fix via the normalisation $\ol \psi = 0$ to obtain a unique solution $\psi$.

Due to the previous considerations, 
we additionally have to demand that the initial data satisfy the charge-neutrality condition 
\begin{equation}
\label{eqinitialconservation}
\int_\Omega \big( n_I - p_I + \varepsilon n_{tr,I} - D \big) dx = 0.
\end{equation}
As a consequence of the structure of system \eqref{eqsystem} and the no-flux boundary conditions in \eqref{eqboundaryconditions}, we see that the total charge is preserved for all $t \geq 0$ in the sense that
\begin{equation}
\label{eqconservationlaw}
\int_\Omega \big( n(t, \cdot) - p(t, \cdot) + \varepsilon n_{tr}(t, \cdot) \big) \, dx = \int_\Omega \big( n_I - p_I + \varepsilon n_{tr,I} \big) \, dx = \int_\Omega D \, dx.
\end{equation}

For the sake of completeness, we subsequently recall all assumptions on our model referring to the introduction above for further details and modelling issues.

\begin{assumption}
\label{assump}
We work on a bounded domain $\Omega \subset \mathbb R^3$ with boundary $\partial \Omega$ of class $C^2$ imposing the following constraints: 
\begin{itemize}
	\item $n_0$, $p_0$, $\tau_n$, $\tau_p$ are positive constants, while $\varepsilon \in (0, \varepsilon_0]$ is bounded by $\varepsilon_0 > 0$,
	\item $D \in L^\infty(\Omega)$ and $V_n, \, V_p \in W^{2,\infty}(\Omega)$ with $\hat n \cdot \nabla V_n = \hat n \cdot \nabla V_p = 0$ on $\partial \Omega$ where $\hat n$ represents the outer unit normal vector on $\partial \Omega$,
	\item $\hat n \cdot J_n = \hat n \cdot J_p = \hat n \cdot \nabla \psi = 0$ on $\partial \Omega$,
	\item $n_I, p_I \in H^1(\Omega) \cap L^\infty(\Omega)$, $n_I, p_I \geq 0$, $0 \leq n_{tr,I} \leq 1$, and \eqref{eqinitialconservation} holds.
\end{itemize}
\end{assumption}

Throughout this article, we suppose Assumption \ref{assump} to hold.

\begin{definition}
A \emph{global weak solution} to \eqref{eqsystem}--\eqref{eqinitialconditions} and \eqref{eqinitialconservation} is a quadruple $(n, p, n_{tr}, \psi) : [0, \infty) \rightarrow H^1(\Omega)^2 \times L^\infty(\Omega) \times H^1(\Omega)$ such that for all $T > 0$ the following conditions are satisfied:
\begin{itemize}
	\item $n, p \in L^2(0, T; H^1(\Omega))$, and
\begin{multline*}
- \int_0^T \langle u'(t), n(t) \rangle_{H^1(\Omega)^\ast \times H^1(\Omega)} dt - \int_\Omega n_I u(0) \, dx \\
= - \int_0^T \int_\Omega J_n(n, \psi) \cdot \nabla u \, dx \, dt + \int_0^T \int_\Omega R_n(n, n_{tr}) u \, dx \, dt,
\end{multline*}
\begin{multline*}
- \int_0^T \langle u'(t), p(t) \rangle_{H^1(\Omega)^\ast \times H^1(\Omega)} dt - \int_\Omega p_I u(0) \, dx \\
= - \int_0^T \int_\Omega J_p(p, \psi) \cdot \nabla u \, dx \, dt + \int_0^T \int_\Omega R_p(p, n_{tr}) u \, dx \, dt,
\end{multline*}
for all $u \in W_2(0, T) \coleq \big\{ f \in L^2(0, T; H^1(\Omega)) \, | \, \partial_t f \in L^2(0, T; H^1(\Omega)^*) \big\}$ subject to $u(T) = 0$,
\item $\begin{aligned}n_{tr}(T) = n_{tr,I} + \frac{1}{\varepsilon} \int_0^T \big( R_p(p, n_{tr}) - R_n(n, n_{tr}) \big) dt\end{aligned}$,
\item 
$\begin{aligned}\lambda \int_\Omega \nabla \psi(T) \cdot \nabla w \, dx = \int_\Omega \big( n(T) - p(T) + \varepsilon n_{tr}(T) - D \big) w \, dx\end{aligned}$
for all $w \in H^1(\Omega)$.
\end{itemize}
\end{definition}

We further mention the embedding $W_2(0, T) \hookrightarrow C([0, T], L^2(\Omega))$ known from PDE theory (see e.g.\ \cite{Chi00}).

\begin{proposition}[Global Solutions]
\label{propglobalsolution}
There exists a unique global weak solution $(n, p, n_{tr}, \psi) : [0, \infty) \rightarrow H^1(\Omega)^2 \times L^\infty(\Omega) \times H^1(\Omega)$ of \eqref{eqsystem}--\eqref{eqinitialconditions} and \eqref{eqinitialconservation} with $\ol {\psi(t, \cdot)} = 0$ for all $t \geq 0$. This solution satisfies $n, p \in L^2(0,T; H^1(\Omega))$ for all $T>0$ uniformly in $\varepsilon \in (0, \varepsilon_0]$ as well as \eqref{eqconservationlaw}. Moreover, $n, p \geq 0$, $0 \leq n_{tr} \leq 1$, and there exist positive constants $M$, $K(M)$ (again uniformly in $\varepsilon \in (0, \varepsilon_0]$) such that
\begin{align}
\label{equpperboundnp}
\| n(t) \|_{L^\infty(\Omega)} + \| p(t) \|_{L^\infty(\Omega)} \leq M \quad \mbox{and} \quad \| \psi(t) \|_{H^2(\Omega)} + \| \psi(t) \|_{C(\ol \Omega)} \leq K
\end{align}
for all $t \geq 0$. In addition, there exists a positive constant $\mu(M, K) < \tfrac12$ (uniformly in $\varepsilon \in (0, \varepsilon_0]$) such that 
\begin{align}
\label{eqlowerboundnp}
n(t, x), \, p(t,x) \geq \mu \min \big\{ t^2, 1 \big\}
\end{align}
and
\begin{align}
\label{eqlowerupperboundntr}
n_{tr}(t, x) \in \big[ \min \big\{ \tfrac{1}{2\varepsilon_0\tau_p} t, \mu \big\}, \, 1 - \min \big\{ \tfrac{1}{2\varepsilon_0\tau_n} t, \mu \big\} \big]
\end{align}
for all $t \geq 0$ and a.e.\ $x \in \Omega$. 
Finally, $n, p \in W_2(0, T) \hookrightarrow C([0, T], L^2(\Omega))$, $n_{tr} \in C([0, T], L^\infty(\Omega))$, and $\psi \in C([0, T], H^2(\Omega))$, where each inclusion holds true uniformly in $\varepsilon \in (0, \varepsilon_0]$.
\end{proposition}

\begin{proposition}[Equilibrium States]
\label{propequilibrium}
The stationary system
\begin{equation}
\label{eqeqsystem}
\begin{cases}
\begin{aligned}
\nabla \cdot J_n(n) + R_n(n,n_{tr}) = 0,& \\
\nabla \cdot J_p(p) + R_p(p,n_{tr}) = 0,& \\
R_n(n, n_{tr}) = R_p(p, \, n_{tr}),& \\
- \lambda \, \Delta \psi = n - p + \varepsilon n_{tr} - D,&
\end{aligned}
\end{cases}
\end{equation}
subject to $\hat n \cdot J_n = \hat n \cdot J_p = \hat n \cdot \nabla \psi = 0$ 
admits a unique solution $(n_\infty, p_\infty, n_{tr,\infty}, \psi_\infty) \in (H^1(\Omega) \cap L^\infty(\Omega))^4$ satisfying $\ol {\psi_\infty} = 0$. The equilibrium potential $\psi_\infty$ is continuous and there exists a positive constant $K_\infty$ such that 
\[
\| \psi_\infty \|_{H^2(\Omega)} + \| \psi_\infty \|_{C(\ol \Omega)} \leq K_\infty.
\]
Moreover, there exist positive constants $M_\infty(K_\infty) > 1$ and $\mu_\infty(K_\infty) < \tfrac12$ such that
\begin{equation}
\label{eqeqbounds}
n_\infty(x), \, p_\infty(x) \in (\mu_\infty, M_\infty) \quad \mbox{and} \quad n_{tr,\infty}(x) \in (\mu_\infty, 1 - \mu_\infty)
\end{equation}
for a.e.\ $x \in \Omega$. The constants $K_\infty$, $M_\infty$, and $\mu_\infty$ are independent of $\varepsilon \in (0, \varepsilon_0]$. 
In detail, the equilibrium densities $n_\infty$, $p_\infty$, and $n_{tr,\infty}$ read 
\begin{equation}
\label{eqeqstates}
n_\infty = n_\ast e^{-\psi_\infty - V_n}, \quad p_\infty = p_\ast e^{\psi_\infty - V_p}, \quad n_{tr,\infty} = \frac{n_\ast}{n_\ast + n_0 e^{\psi_\infty}} = \frac{p_0}{p_0 + p_\ast e^{\psi_\infty}}
\end{equation}
where the positive constants $n_\ast$ and $p_\ast$ are uniquely determined in terms of $\psi_\infty$ by 
\begin{equation}
\label{eqeqconstants}
n_\ast p_\ast = n_0 p_0 \quad \mbox{and} \quad n_\ast \ol{e^{-\psi_\infty - V_n}} - p_\ast \ol{e^{\psi_\infty - V_p}} + \varepsilon \ol{\frac{n_\ast}{n_\ast + n_0 e^{\psi_\infty}}} - \ol D = 0.
\end{equation}
Furthermore, the following relations hold true: 
\begin{equation}
\label{eqntrrelations}
n_{tr,\infty} = \frac{n_\infty (1 - n_{tr,\infty})}{n_0 \mu_n} \quad \mbox{and} \quad 1 - n_{tr,\infty} = \frac{p_\infty n_{tr,\infty}}{p_0 \mu_p}.
\end{equation}
\end{proposition}

We introduce the \emph{entropy functional} $E(n, p, n_{tr}, \psi)$ for non-negative functions $n, p, n_{tr} \in L^2(\Omega)$ satisfying $n_{tr} \leq 1$ and $\ol n - \ol p + \varepsilon \ol{n_{tr}} = \ol D$ where $\psi \in H^1(\Omega)$ is the unique solution of \eqref{eqpsi} with right hand side $f = n - p + \varepsilon n_{tr} - D$ and normalisation $\ol \psi = 0$:
\begin{multline}
E(n, p, n_{tr}, \psi) \coleq \int_{\Omega} \bigg( n \ln \frac{n}{n_0 \mu_n} - (n-n_0\mu_n) \\ + p \ln \frac{p}{p_0 \mu_p} - (p-p_0\mu_p) + \frac{\lambda}{2} \big| \nabla \psi \big|^2 + \varepsilon \int_{1/2}^{n_{tr}} \ln \left( \frac{s}{1-s} \right) ds \bigg) dx. \label{eqentropy}
\end{multline}
The densities $n$ and $p$ enter via Boltzmann entropy contributions $a \ln a - (a - 1) \geq 0$, whereas $n_{tr}$ appears within the entropy functional via an integral term. We first mention that the integral $\int_{1/2}^{n_{tr}} \ln \big( \frac{s}{1-s} \big) ds$ is non-negative and finite for all $n_{tr}(x) \in [0,1]$. In more detail, we may write 
\begin{multline*}
\int_{1/2}^{n_{tr}} \ln \left( \frac{s}{1-s} \right) ds = \big[ n_{tr} \ln n_{tr} - (n_{tr} - 1) \big] \\
+ \big[ (1 - n_{tr}) \ln (1 - n_{tr}) - ((1 - n_{tr}) - 1) \big] + \ln 2 - 1.
\end{multline*}
Consequently, both the occupied and unoccupied trapped states ($n_{tr}$ and $1 - n_{tr}$) are described via Boltzmann statistics within the entropy functional, and the integral $\int_{1/2}^{n_{tr}} \ln \big( \frac{s}{1-s} \big) ds$ allows to combine the contributions of $n_{tr}$ and $1 - n_{tr}$ in a compact fashion.

One can further verify that the entropy functional \eqref{eqentropy} is indeed a Lyapunov functional: By defining the \emph{entropy production functional}
\begin{equation}
\label{eeplaw}
P(n, p, n_{tr}, \psi) := -\frac{d}{dt} E(n, p, n_{tr}, \psi),
\end{equation}
we first calculate along solutions of \eqref{eqsystem} that formally
\begin{multline}
P(n, p, n_{tr}, \psi) = 
\int_{\Omega} \bigg( \frac{|J_n|^2}{n} + \frac{|J_p|^2}{p} \\ - R_n \ln \left( \frac{n(1-n_{tr})}{n_0 \mu_n n_{tr}} \right) - R_p \ln \left( \frac{p n_{tr}}{p_0 \mu_p (1-n_{tr})} \right) \bigg) dx. \label{eqproduction}
\end{multline}
The entropy production functional involves non-negative flux terms as well as recombination terms of the form $(a-1) \ln a \geq 0$. The entropy production $P$ is, therefore, a non-negative functional, which ensures the monotone decrease in time of the entropy $E$ along trajectories of \eqref{eqsystem}. More precisely, it will be shown in Theorem \ref{theoremexpconvergence} that the global weak solutions to \eqref{eqsystem} obtained in Proposition \ref{propglobalsolution} satisfy a suitable weak version of \eqref{eeplaw}, see \eqref{eqweakeplaw} below.

The following theorem constitutes a so-called entropy--entropy production (EEP) estimate. This is a functional inequality between entropy and entropy production for arbitrary, yet admissible non-negative functions $n, p, n_{tr} \in L^\infty(\Omega)$, $n_{tr} \leq 1$; in particular, the electrostatic potential $\psi \in H^1(\Omega)$ in the following theorem must be the unique solution of \eqref{eqpsi} subject to $f = n - p + \varepsilon n_{tr} - D$ and the normalisation $\ol \psi = 0$.

\begin{theorem}[Entropy--Entropy Production Estimate]
\label{theoremeepinequality}
Consider all non-negative functions $n, p, n_{tr} \in L^\infty(\Omega)$ subject to $n, p \leq \mathcal{M}$, $n_{tr} \leq 1$, and $\ol n - \ol p + \varepsilon \ol{n_{tr}} = \ol D$ and accordingly determine $\psi \in H^1(\Omega)$ as the unique solution to \eqref{eqpsi} with $f = n - p + \varepsilon n_{tr} - D$ and $\ol \psi = 0$. Then, there exist explicit constants $\varepsilon_0 > 0$ and $C_{EEP} > 0$ depending on $\mathcal{M}$ and on $K_\infty$ (as given in Proposition \ref{propequilibrium})  
such that 
\[
E(n, p, n_{tr}, \psi) - E(n_\infty, p_\infty, n_{tr,\infty}, \psi_\infty) \leq C_{EEP} P(n, p, n_{tr}, \psi)
\]
holds true for all $\varepsilon \in (0, \varepsilon_0]$. 
\end{theorem}

%

Note that $C_{EEP}$ is independent of $\varepsilon \in (0, \varepsilon_0]$ and that this abstract EEP inequality can be applied to the global solution to \eqref{eqsystem} by using $\mathcal{M}=M$ from Proposition \ref{propglobalsolution}.
We are then able to prove the exponential decay of the entropy relative to the equilibrium by using a Gronwall argument.

\begin{theorem}[Exponential Decay of the Relative Entropy]
\label{theoremexpconvergence}
Let $\varepsilon \in (0, \varepsilon_0]$ with $\varepsilon_0 > 0$ from Theorem \ref{theoremeepinequality}, and let $(n, p, n_{tr}, \psi)$ be the unique global weak solution to \eqref{eqsystem} with non-negative initial datum $(n_I, p_I, n_{tr,I}) \in (H^1(\Omega) \cap L^\infty(\Omega))^2 \times L^\infty(\Omega)$, $n_{tr,I} \leq 1$, satisfying $\ol{n_I} - \ol{p_I} + \varepsilon \ol{n_{tr,I}} = \ol D$ according to Proposition \ref{propglobalsolution}. In addition, let $(n_\infty, p_\infty, n_{tr,\infty}, \psi_\infty)$ be the unique equilibrium state characterised in Proposition \ref{propequilibrium} as a solution to \eqref{eqeqsystem}. Then, $(n, p, n_{tr}, \psi)$ fulfils the weak entropy production law 
\begin{equation}
\label{eqweakeplaw}
E(n,p,n_{tr},\psi)(t_1) + \int_{t_0}^{t_1} P(n, p, n_{tr}, \psi)(s) \, ds = E(n,p,n_{tr},\psi)(t_0)
\end{equation}
for all $0 < t_0 \leq t_1 < \infty$. As a consequence, $E(n, p, n_{tr}, \psi)$ converges exponentially to $E(n_\infty, p_\infty, n_{tr,\infty}, \psi_\infty)$ with explicit rate and constant as a function of time $t \geq 0$. More precisely, 
\begin{multline}
\label{eqexpconventropy}
E(n, p, n_{tr}, \psi) - E(n_\infty, p_\infty, n_{tr,\infty}, \psi_\infty) \\
\leq \Big( E(n_I, p_I, n_{tr,I}, \psi_I) - E(n_\infty, p_\infty, n_{tr,\infty}, \psi_\infty) \Big) e^{-C_{EEP}^{-1} t}
\end{multline}
where $\psi_I \in H^1(\Omega)$ is the unique weak solution to \eqref{eqpsi} with $f = n_I - p_I + \varepsilon n_{tr,I} - D$ and $\ol{\psi_I} = 0$.
\end{theorem}

\begin{corollary}[Exponential Convergence to the Equilibrium]
\label{corconvlinfty}
Under the hypotheses of Theorem \ref{theoremexpconvergence}, the following improved convergence properties with constants $0 < c, C < \infty$ both depending on $M$ and $K_\infty$ but not on $\varepsilon \in (0, \varepsilon_0]$ hold true for all $t \geq 0$:
\begin{multline}
\label{eqexpconvlinfty}
\| n - n_\infty \|_{L^\infty(\Omega)} + \| p - p_\infty \|_{L^\infty(\Omega)} \\ + \| n_{tr} - n_{tr,\infty} \|_{L^\infty(\Omega)} + \| \psi - \psi_\infty \|_{H^2(\Omega)} \leq C e^{-c\, t}.
\end{multline}
In particular, $\psi \rightarrow \psi_\infty$ in $L^\infty(\Omega)$ at an exponential rate.
\end{corollary}

The remainder of this article is devoted to the proofs of the various statements above. Section \ref{secsolution} collects the proofs of Propositions \ref{propglobalsolution} and \ref{propequilibrium}. The proof of Theorem \ref{theoremeepinequality} along with the necessary prerequisites is contained in Section \ref{seceep}, while the results on exponential convergence to equilibrium are proven in Section \ref{secconv}. A brief section providing an outlook to future research concludes the paper.

\section{Global solution and equilibrium state}
\label{secsolution}
\begin{proof}[Proof of Proposition \ref{propglobalsolution}]
	The existence of such a unique global solution $(n, p, n_{tr}, \psi)$ as well as $(n, p) \in L^2(0,T; H^1(\Omega))$ for all $T>0$ uniformly for $\varepsilon \in (0,\varepsilon_0]$ and $0 \leq n_{tr} \leq 1$ are a consequence of \cite[Lemma 3.1]{GMS07}.
	The uniform-in-time $L^\infty$ bounds for $n$ and $p$ follow similar to \cite[Lemma 4.1]{DFM08}, where a Nash--Moser-type iteration for $L^r$ norms, $r \geq 1$, of $n$ and $p$ has been employed. But as the coupling to Poisson's equation is missing in \cite{DFM08}, we have to slightly modify the line of arguments.
	
	The evolution of the $L^{r+1}$ norm, $r \geq 1$, of $n$ and $p$ can be reformulated as
	\begin{align*}
	\frac{d}{dt} &\int_\Omega \big( n^{r+1} + p^{r+1} \big) \, dx = (r+1) \int_\Omega \Big( -r n^{r-1} \nabla n \cdot \big(\nabla n + n \nabla (\psi + V_n) \big) \\
	&- r p^{r-1} \nabla p \cdot \big(\nabla p + p \nabla (-\psi + V_p) \big) + n^r R_n + p^r R_p \Big) dx \\
	\leq&- \frac{4r}{r+1} \int_\Omega \Big| \nabla n^\frac{r+1}{2} \Big|^2 \, dx -\frac{4r}{r+1} \int_\Omega \Big| \nabla p^\frac{r+1}{2} \Big|^2 \, dx + r \int_\Omega n^{r+1} \Delta (\psi + V_n) \, dx \\
	&+ r \int_\Omega p^{r+1} \Delta (-\psi + V_p) \, dx + (r+1) \int_\Omega \big( n^r R_n + p^r R_p \big) \, dx.
	\end{align*}
	To the last term, we apply the estimate $(r+1) n^r \leq \frac1r + 2 r n^{r+1}$ which follows from Young's inequality $ab \leq \frac{1}{q} a^q + \frac{1}{s} b^s$ with $a \coleq (\frac{1}{r+1})^\frac{1}{r+1}$, $b \coleq (\frac{1}{r+1})^{-\frac{1}{r+1}} n^r$, $q \coleq r+1$, and $s \coleq \frac{r+1}{r}$.
	As a consequence of $-\lambda \Delta \psi = n - p + \varepsilon n_{tr} - D$, $(n^{r+1} - p^{r+1})(n-p) \geq 0$, $|R_n| \leq C 
	(1 + n)$, and $|R_p| \leq C 
	(1 + p)$, we then deduce
	\begin{align}
	\label{eqnashmoser1}
	\frac{d}{dt} \int_\Omega \big( n^{r+1} + p^{r+1} \big) \, dx \leq &-\frac{4r}{r+1} \int_\Omega \Big| \nabla n^\frac{r+1}{2} \Big|^2 \, dx -\frac{4r}{r+1} \int_\Omega \Big| \nabla p^\frac{r+1}{2} \Big|^2 \, dx \\
	&+ \wh C r \int_\Omega \big( n^{r+1} + p^{r+1} \big) \, dx + \frac{\wh C}{r} \nonumber
	\end{align}
	with a constant $\wh C > 0$ depending on 
	$\varepsilon_0$ but not on $r$. One can now proceed as in \cite[Lemma 4.1]{DFM08}. For completeness, we briefly collect the main arguments below and refer to \cite{DFM08} for the details. By utilizing the Gagliardo--Nirenberg-type inequality $\| f \|_{L^2(\Omega)} \leq C_{GN} \| f \|_{L^1(\Omega)}^\frac{2}{5} \| f \|_{H^1(\Omega)}^\frac{3}{5}$ for $f \coleq n^\frac{r+1}{2}$ and $f \coleq p^\frac{r+1}{2}$, one derives 
	\begin{multline}
	\label{eqnashmoser2}
	\int_\Omega \big( n^{r+1} + p^{r+1} \big) \, dx \\ \leq \delta \int_\Omega \bigg( \Big| \nabla n^\frac{r+1}{2} \Big|^2 + \Big| \nabla p^\frac{r+1}{2} \Big|^2 \bigg) dx + \frac{\wt C}{\delta} \Bigg( \int_\Omega \Big( n^\frac{r+1}{2} + p^\frac{r+1}{2} \Big) \, dx \Bigg)^2
	\end{multline}
	where $\wt C > 0$ is a constant independent of $r$ and $\delta > 0$. We now introduce $\lambda_k \coleq 2^k - 1$ for $k \geq 1$ and set $r \coleq \lambda_k$. Choosing a sufficiently small constant $A > 0$ and defining $\delta_k \coleq \frac{A}{\lambda_k}$ results in 
	\[
	\delta_k \big(\wh C \lambda_k + \delta_k \big) \leq \frac{4 \lambda_k}{\lambda_k + 1}
	\]
	for all $k \geq 1$. By multiplying \eqref{eqnashmoser2} with $\wh C \lambda_k + \delta_k$ and by combining the result with \eqref{eqnashmoser1}, we arrive at 
	\begin{multline*}
	\frac{d}{dt} \int_\Omega \big( n^{\lambda_k + 1} + p^{\lambda_k + 1} \big) \, dx \leq - \delta_k \int_\Omega \big( n^{\lambda_k + 1} + p^{\lambda_k + 1} \big) \, dx \\
	+ B \lambda_k (\lambda_k + \delta_k) \sup_{0 \leq \tau \leq t} \Bigg( \int_\Omega \Big( n^\frac{r+1}{2} + p^\frac{r+1}{2} \Big) \, dx \Bigg)^2 + \frac{\wh C}{\lambda_k}
	\end{multline*}
	with the constant $B \coleq \frac{\wh C \wt C}{A}$. The uniform $L^\infty$ bounds on $n$ and $p$ now follow from \cite[Lemma 4.2]{DFM08}.
	
	As, in particular, $\| n(t) - p(t) + \varepsilon n_{tr}(t) - D \|_{L^2(\Omega)}$ is uniformly bounded in $t \geq 0$, we conclude that $\| \psi(t) \|_{H^2(\Omega)}$ is uniformly bounded in time by applying standard elliptic regularity theory (see e.g.\ \cite[Chap. IV. \textsection 2. Theorem 4]{Mik80}). The announced bound on $\| \psi(t) \|_{C(\ol \Omega)}$ follows from the embedding $H^2(\Omega) \hookrightarrow C(\ol \Omega)$ valid in $\mathbb R^3$.
	
	The regularity $\partial_t n, \partial_t p \in L^2(0, T; H^{1}(\Omega)^\ast)$ and, hence, $n, p \in W_2(0, T) \hookrightarrow C([0, T], L^2(\Omega))$ uniformly for $\varepsilon \in (0,\varepsilon_0]$ is easily inferred from the corresponding bounds on $J_n, J_p \in L^2((0, T) \times \Omega)$ and $R_n, R_p \in L^\infty((0, T) \times \Omega)$. Likewise, $n_{tr} \in C([0, T], L^\infty(\Omega))$ and $\psi \in C([0, T], H^2(\Omega))$, where both inclusions hold true uniformly for $\varepsilon \in (0,\varepsilon_0]$.
	
	For showing the upper and lower bound on $n_{tr}$ in \eqref{eqlowerupperboundntr}, we multiply the third equation in \eqref{eqsystem} with $\tau_p$ and observe that
	\[
	\varepsilon \partial_t (\tau_p n_{tr}) \geq 1 - \rho n_{tr}
	\]
	holds true with a constant $\rho(M) > 1$ due to $\| p \|_{L^\infty(\Omega)} \leq M$.
	We now distinguish the following three cases for all $t \geq 0$ and a.e.\ $x \in \Omega$: $n_{tr}(t, x) \geq \tfrac{1}{\rho}$, $n_{tr}(t, x) \in [\frac{1}{2 \rho}, \tfrac{1}{\rho})$, and $n_{tr}(t, x) < \tfrac{1}{2 \rho}$. In the first case, $\partial_t (\tau_p n_{tr}(t,x)) \leq 0$, while in the second case $\partial_t (\tau_p n_{tr}(t,x)) > 0$. And in the third case, $\partial_t (\tau_p n_{tr}(t,x)) > \tfrac{1}{2 \varepsilon_0}$. Defining $t_0 \coleq \tfrac{\varepsilon_0 \tau_p}{\rho}$, this ensures 
	\begin{equation}
	\label{eqlowerboundntr}
	\tau_p n_{tr}(t,x) \geq \frac{t}{2 \varepsilon_0}, \quad t \in [0, t_0], \qquad \mbox{and} \qquad  n_{tr}(t,x) \geq \frac{1}{2 \rho}, \quad t \geq t_0.
	\end{equation}
	The upper bound on $n_{tr}$ follows by applying the same arguments to $\tau_n(1 - n_{tr})$.
	
	Concerning the bounds on $n$ and $p$ in \eqref{eqlowerboundnp}, we follow the lines in \cite{FK20} and concentrate on the arguments for $n$ as the result for $p$ can be derived analogously. 
	For simplicity, we set w.l.o.g.\ $\tau_n = \tau_p = 1$
	in the following calculations.	
	The temporal derivative of $n$ is then bounded from below by 
	\begin{align}
	\label{eqdtw}
	\partial_t n \geq \nabla \cdot \big( \nabla n + n \nabla (\psi + V_n) \big) + n_{tr} - \alpha n
	\end{align}
	with a constant $\alpha > 0$. 
	Employing 
	the no-flux boundary conditions from \eqref{eqboundaryconditions}, we first test \eqref{eqdtw} with $\bigl( n - \mu_1 t^2 \bigr)_-$ for $t \in [0, t_0]$ where $\mu_1 > 0$ is a constant specified below and where we abbreviate $(\cdot)_{-} \coleq \min\{\cdot,0\}$. This entails
	\begin{align*}
	\frac{d}{dt} \frac{1}{2} &\int_{\Omega} \Bigl(n - \mu_1 t^2\Bigr)_{-}^2 \, dx = \int_{\Omega} \Bigl(n - \mu_1 t^2\Bigr)_{-} \big( \partial_t n - 2\mu_1 t \big) \, dx	\\
	\leq &\int_{\Omega} \Bigl(n-\mu_1 t^2\Bigr)_{-} \Big( \nabla \cdot \big(\nabla n + n \nabla (\psi + V_n) \big) + n_{tr} - \alpha n - 2 \mu_1 t \Big) \, dx \\
	= &- \int_{\Omega} \mathbb{1}_{n \leq \mu_1 t^2} \nabla n \cdot \big( \nabla n + (n - \mu_1 t^2) \nabla (\psi + V_n) + \mu_1 t^2 \nabla (\psi + V_n) \big) \, dx \\
	&+ \int_{\Omega} \Big( n-\mu_1 t^2 \Big)_{-} \big(n_{tr} - \alpha n - 2\mu_1 t \big) \, dx.
	\end{align*}
	Omitting the first term on the right hand side, we further derive 
	\begin{align*}
	\frac{d}{dt} \frac{1}{2} \int_{\Omega} \Bigl(n - \mu_1 t^2\Bigr)_{-}^2 \, dx \leq &-\frac12 \int_\Omega \nabla \bigg[ \Bigl(n-\mu_1 t^2\Bigr)_{-}^2 \bigg] \cdot \nabla (\psi + V_n) \, dx \\ &- \int_\Omega \nabla \bigg[ \Bigl(n-\mu_1 t^2\Bigr)_{-} \bigg] \cdot \mu_1 t^2 \nabla (\psi + V_n) \, dx \\ &+ \int_{\Omega} \Big( n-\mu_1 t^2 \Big)_{-} \big(n_{tr} - \alpha n - 2\mu_1 t \big) \, dx.
	\end{align*}
	We now integrate by parts utilizing $\hat n \cdot \nabla \psi = \hat n \cdot \nabla V_n = 0$ on $\partial \Omega$. Due to the bound $n_{tr}(t,x) \geq \tfrac{t}{2 \varepsilon_0}$ on the considered interval $t \in [0, t_0]$, we obtain 
	\begin{multline*}
	\frac{d}{dt} \frac{1}{2} \int_{\Omega} \Bigl(n - \mu_1 t^2\Bigr)_{-}^2 \, dx \leq \frac12 \int_\Omega \Big( n-\mu_1 t^2 \Big)_{-}^2 \| \Delta (\psi + V_n) \|_{L^\infty(\Omega)} \, dx \\
	+ \int_{\Omega} \Big( n-\mu_1 t^2 \Big)_{-} \Big(\frac{1}{2 \varepsilon_0} - \alpha \mu_1 t_0 - 2 \mu_1 - \mu_1 t_0 \| \Delta (\psi + V_n) \|_{L^\infty(\Omega)} \Big) t \, dx.
	\end{multline*}
	Choosing $\mu_1 > 0$ according to $\mu_1 \bigl( 2 + \alpha t_0 + t_0 \| \Delta (\psi + V_n) \|_{L^\infty(\Omega)} \bigr) \leq \tfrac{1}{2 \varepsilon_0}$, we deduce
	\[
	\frac{d}{dt} \int_{\Omega} \Bigl(n-\mu_1 t^2\Bigr)_{-}^2 \, dx 
	\leq \|\Delta (\psi + V_n) \|_{L^\infty(\Omega)} \int_\Omega \Bigl(n-\mu_1 t^2\Bigr)_{-}^2 \, dx.
	\]
	Because of $\int_{\Omega} \bigl(n(0,x)\bigr)_{-}^2 \, dx = 0$, we derive $\int_{\Omega} \big(n - \mu_1 t^2 \big)_{-}^2 \, dx = 0$ for all $t \in [0, t_0]$ by applying a Gronwall argument. We thus arrive at $n(t,x) \geq \mu_1 t^2$ for all $t \in [0, t_0]$ and a.e.\ $x\in\Omega$.
	
	In the situation $t \geq t_0$, we test \eqref{eqdtw} with $( n - \mu_2 )_-$ where $\mu_2 > 0$ is another constant to be specified. As above, we calculate
	\begin{align*}
	\frac{d}{dt} \frac{1}{2} &\int_{\Omega} \bigl(n - \mu_2\bigr)_{-}^2 \, dx = \int_{\Omega} \bigl(n - \mu_2\bigr)_{-} \partial_t n \, dx	\\
	\leq &\int_{\Omega} \bigl(n-\mu_2\bigr)_{-} \Big( \nabla \cdot \big(\nabla n + n \nabla (\psi + V_n) \big) + n_{tr} - \alpha n \Big) \, dx \\
	= &- \int_{\Omega} \mathbb{1}_{n \leq \mu_2} \nabla n \cdot \big( \nabla n + (n - \mu_2) \nabla (\psi + V_n) + \mu_2 \nabla (\psi + V_n) \big) \, dx \\
	&+ \int_{\Omega} \big( n-\mu_2 \big)_{-} \big(n_{tr} - \alpha n \big) \, dx.
	\end{align*}
	The same reasoning as above gives rise to 
	\begin{multline*}
	\frac{d}{dt} \frac{1}{2} \int_{\Omega} \bigl(n - \mu_2\bigr)_{-}^2 \, dx \leq -\frac12 \int_\Omega \nabla \Big[ \bigl(n-\mu_2\bigr)_{-}^2 \Big] \cdot \nabla (\psi + V_n) \, dx \\ - \int_\Omega \nabla \Big[ \bigl(n-\mu_2\bigr)_{-} \Big] \cdot \mu_2 \nabla (\psi + V_n) \, dx + \int_{\Omega} \big( n-\mu_2 \big)_{-} \big(n_{tr} - \alpha n \big) \, dx.
	\end{multline*}
	For $t \geq t_0$, we have the lower bound $n_{tr}(t,x) \geq \frac{1}{2 \rho}$, which yields
	\begin{multline*}
	\frac{d}{dt} \frac{1}{2} \int_{\Omega} \big( n - \mu_2 \big)_{-}^2 \, dx \leq \frac12 \int_\Omega \big( n - \mu_2 \big)_{-}^2 \| \Delta (\psi + V_n) \|_{L^\infty(\Omega)} \, dx \\
	+ \int_{\Omega} \big( n - \mu_2 \big)_{-} \Big( \frac{1}{2\rho} - \alpha \mu_2 - \mu_2 \| \Delta (\psi + V_n) \|_{L^\infty(\Omega)} \Big) \, dx.
	\end{multline*}
	If we impose the conditions $\mu_2 (\alpha + \| \Delta (\psi + V_n) \|_{L^\infty(\Omega)}) \leq \frac{1}{2 \rho}$ and $\mu_2 \leq \mu_1 t_0^2$ on $\mu_2 > 0$, we infer
	\[
	\frac{d}{dt} \int_{\Omega} \big( n - \mu_2 \big)_{-}^2 \, dx \leq \|\Delta (\psi + V_n)\|_{L^\infty(\Omega)} \int_{\Omega} \big( n - \mu_2 \big)_{-}^2 \, dx
	\]
	as well as $\int_{\Omega} \big( n(t_0,x) - \mu_2 \big)_{-}^2 \, dx = 0$. Finally, Gronwall's lemma guarantees that $\int_{\Omega} \big( n - \mu_2 \big)_{-}^2 \, dx = 0$ and, hence, $n(t,x) \geq \mu_2$ for all $t \geq t_0$ and a.e.\ $x \in \Omega$.
\end{proof}

\begin{proof}[Proof of Proposition \ref{propequilibrium}]
	As the entropy production vanishes at the stationary state $(n_\infty, p_\infty, n_{tr,\infty}, \psi_\infty)$, straightforward calculations show that $J_n = J_p = R_n = R_p = 0$ yields the representations for $n_\infty$ and $p_\infty$ as well as the two expressions for $n_{tr,\infty}$ in \eqref{eqeqstates}. The relation $n_\ast p_\ast = n_0 p_0$ results from a combination of the formulas arising from $R_n = 0$ and $R_p = 0$, while the fact that the conservation law is also fulfilled in the equilibrium follows from integrating Poisson's equation in \eqref{eqeqsystem}. Note that $n_\ast$ (and hence $p_\ast$) is uniquely determined from the second relation in \eqref{eqeqconstants} as its left hand side is strictly monotonously increasing and surjective from $(0, \infty)$ to $(-\infty,\infty)$ as a function of $n_\ast$. The identities in \eqref{eqntrrelations} are equivalent versions of $R_n(n_\infty, n_{tr,\infty}) = 0$ and $R_p(p_\infty, n_{tr,\infty}) = 0$. 
	
	Next, we establish the existence of the limiting potential $\psi_\infty$. A technical difficulty stems from the 
	fact that the constant $n_\ast$ (or equally $p_\ast$) depends non-locally on $\psi_\infty$, see \eqref{eqeqconstants}.
	However, this can be avoided by substituting $\wt{\psi_\infty} \coleq \psi_\infty - \ln n_\ast$ and rewriting $- \lambda \, \Delta \psi_\infty = n_\infty - p_\infty + \varepsilon n_{tr,\infty} - D$ as 
	\begin{align}
	\label{eqeqpotential}
	- \lambda \, \Delta \wt{\psi_\infty} - e^{-\wt{\psi_\infty} - V_n} + n_0 p_0 e^{\wt{\psi_\infty} - V_p} - \frac{\varepsilon}{1 + n_0 e^{\wt{\psi_\infty}}} = - D.
	\end{align}
	We now aim to apply 
	\cite[Theorem 4.8]{T10} to \eqref{eqeqpotential}, which we further reformulate as
	\begin{align}
	\label{eqeqpotentialabstract}
	- \lambda \, \Delta \wt{\psi_\infty} + \min\{ e^{-V_n}, n_0 p_0 e^{-V_p} \} \wt{\psi_\infty} + d(\cdot, \wt{\psi_\infty}) = - D
	\end{align}
	where 
	\begin{align*}
	d(x,y) \coleq - e^{-y-V_n} + n_0 p_0 e^{y-V_p} - \frac{\varepsilon}{1 + n_0 e^{y}} - \min\{ e^{-V_n}, n_0 p_0 e^{-V_p} \} y.
	\end{align*}
The structure of \eqref{eqeqpotentialabstract} is suitable to apply
\cite[Theorem 4.8]{T10} for the existence of a unique continuous solution
provided that $d$ is monotone increasing w.r.t.\ $y$: Indeed, direct computations show 
\begin{equation*}
e^{-y-V_n} + n_0 p_0 e^{y-V_p} \ge 2 \sqrt{n_0 p_0} e^{-\frac{V_p + V_n}{2}}, \qquad \forall y\in\mathbb{R},
\end{equation*}
(where the lower bound is attained at the unique minimum $e^{y}=e^{\frac{V_p-V_n}{2}}/\sqrt{n_0 p_0}$).
Hence, we estimate independently of $\varepsilon$ 
\begin{align*}
\partial_y d(x,y) &\ge e^{-y-V_n} + n_0 p_0 e^{y-V_p}  - \min\{ e^{-V_n}, n_0 p_0 e^{-V_p} \} \\
&\ge  2 e^{-\frac{V_n}{2}} \sqrt{n_0 p_0} e^{-\frac{V_p}{2}} - \min\{ e^{-V_n}, n_0 p_0 e^{-V_p} \} > 0, 
\end{align*}
and therefore strict monotonicity of $d$ w.r.t.\ $y$ follows.

%
	
	As a consequence, \eqref{eqeqpotentialabstract} admits a unique solution $\wt{\psi_\infty} \in H^1(\Omega) \cap L^\infty(\Omega)$, which is continuous on $\ol \Omega$ and bounded via
	\[
	\| \wt{\psi_\infty} \|_{H^1(\Omega)} + \| \wt{\psi_\infty} \|_{C(\ol \Omega)} \leq \wt{K_\infty}
	\]
	where the constant $\wt{K_\infty}$ is independent of $\varepsilon$. Going back to $\psi_\infty$, the constraint $\ol{\psi_\infty} = 0$ implies 
	\[
	n_\ast = e^{-\int_\Omega \wt{\psi_\infty} \, dx},
	\]
	which in turn uniquely determines $\psi_\infty = \wt{\psi_\infty} - \int_\Omega \wt{\psi_\infty} \, dx$. The bounds on $\wt{\psi_\infty}$ directly transfer to $\psi_\infty$.
	
	As in \cite{FK18}, we verify the bounds \eqref{eqeqbounds} by solving the two equations in \eqref{eqeqconstants} for $n_\ast > 0$ abbreviating $V_\infty \coleq \max \{ \| V_n \|_{L^\infty(\Omega)}, \| V_p \|_{L^\infty(\Omega)} \}$:
	\begin{align*}
	n_\ast &= \frac{\ol{D - \varepsilon n_{tr,\infty}}}{2\ol{e^{-\psi_\infty-V_n}}} + \sqrt{\frac{\ol{D - \varepsilon n_{tr,\infty}}^2}{4\ol{e^{-\psi_\infty-V_n}}^2} + n_0 p_0 \frac{\ol{e^{\psi_\infty-V_p}}}{\ol{e^{-\psi_\infty-V_n}}}} \\
	&\leq e^{K_\infty + V_\infty} (\sqrt{n_0p_0} + \varepsilon_0 + | \ol D |).
	\end{align*}
	We stress that the same bound is valid also for $p_\ast > 0$,
	and that the upper and lower bounds on $n_\infty$, $p_\infty$ and $n_{tr,\infty}$ are a consequence of the bounds on $n_\ast$ and $p_\ast$ as well as $n_\ast p_\ast = n_0 p_0$. Finally, the estimate
	\[
	\| \psi_\infty \|_{H^2(\Omega)} \leq C \| n_\infty - p_\infty + \varepsilon n_{tr,\infty} - D \|_{L^2(\Omega)} \leq C(K_\infty)
	\]
	ensures the higher regularity of $\psi_\infty$.
\end{proof}

\section{Derivation of an EEP inequality}
\label{seceep}
As an auxiliary result, we first derive a convenient expression for the entropy relative to the equilibrium.
\begin{lemma}
\label{lemmarelativeentropy}
The entropy relative to the equilibrium equals
\begin{align*}
&E(n,p,n_{tr},\psi) - E(n_\infty,p_\infty,n_{tr,\infty},\psi_\infty) \\
&\qquad = \int_{\Omega} \bigg( n \ln \frac{n}{n_\infty} - (n-n_\infty) + p \ln \frac{p}{p_\infty} - (p-p_\infty) \\
&\qquad\quad + \frac{\lambda}{2} \big| \nabla (\psi - \psi_\infty) \big|^2 + \varepsilon \int_{n_{tr,\infty}}^{n_{tr}} \left( \ln \frac{s}{1-s} - \ln \frac{n_{tr,\infty}}{1-n_{tr,\infty}} \right) ds \bigg) dx.
\end{align*}
\end{lemma}

\begin{proof}
According to the definition of $E(n,p,n_{tr},\psi)$, one has 
\begin{align*}
&E(n,p,n_{tr}, \psi) - E(n_\infty, p_\infty, n_{tr,\infty}, \psi_\infty) \\
&= \int_\Omega \bigg( n \ln \frac{n}{n_0 \mu_n}\!-\!n_\infty \ln \frac{n_\infty}{n_0 \mu_n}\!-\!(n\!-\!n_\infty) + p \ln \frac{p}{p_0 \mu_p}\!-\!p_\infty \ln \frac{p_\infty}{p_0 \mu_p}\!-\!(p\!-\!p_\infty) \\
&\quad + \frac{\lambda}{2} \big( |\nabla \psi|^2 - |\nabla \psi_\infty|^2 \big) + \varepsilon \int_{n_{tr,\infty}}^{n_{tr}} \ln \frac{s}{1-s} ds \bigg) dx.
\end{align*}
We rewrite the first integrand as
$
n \ln \frac{n}{n_0 \mu_n} = n \ln \frac{n}{n_\infty} + n \ln \frac{n_\infty}{n_0 \mu_n}
$
and use $\frac{n_\infty}{n_0 \mu_n} = \frac{n_\ast}{n_0} e^{-\psi_\infty}$ to find
\begin{multline*}
\int_\Omega \bigg( n \ln \frac{n}{n_0 \mu_n} - n_\infty \ln \frac{n_\infty}{n_0 \mu_n} - (n - n_\infty) \bigg) dx \\ = \int_\Omega \bigg( n \ln \frac{n}{n_\infty} - (n - n_\infty) + (n - n_\infty) \Big(\ln \frac{n_\ast}{n_0} - \psi_\infty \Big) \bigg) dx.
\end{multline*}
Together with an analogous calculation for the $p$-terms, we obtain
\begin{align*}
&E(n,p,n_{tr}, \psi) - E(n_\infty, p_\infty, n_{tr,\infty}, \psi_\infty) \\
&\quad = \int_\Omega \bigg( n \ln \frac{n}{n_\infty} - (n - n_\infty) + p \ln \frac{p}{p_\infty} - (p - p_\infty) \\
&\qquad + (n - n_\infty) \Big(\ln \frac{n_\ast}{n_0} - \psi_\infty \Big) + (p - p_\infty) \Big(\ln \frac{p_\ast}{p_0} + \psi_\infty \Big) \\
&\qquad + \frac{\lambda}{2} |\nabla \psi|^2 - \frac{\lambda}{2} |\nabla \psi_\infty|^2 + \varepsilon \int_{n_{tr,\infty}}^{n_{tr}} \ln \frac{s}{1-s} \, ds \bigg) dx.
\end{align*}
We now employ the conservation law $\ol p - \ol{p_\infty} = \ol n - \ol{n_\infty} + \varepsilon (\ol{n_{tr}} - \ol{n_{tr,\infty}})$, the formula $n_\ast p_\ast = n_0 p_0$ and the representation $\frac{p_\ast}{p_0} = \frac{1 - n_{tr,\infty}}{n_{tr,\infty}} e^{-\psi_\infty}$ to derive
\begin{align*}
&(\ol n - \ol{n_\infty}) \ln \frac{n_\ast}{n_0} + (\ol p - \ol{p_\infty}) \ln \frac{p_\ast}{p_0} \\
&\quad = (\ol n - \ol{n_\infty}) \ln \frac{n_\ast p_\ast}{n_0 p_0} + \varepsilon \int_\Omega (n_{tr} - n_{tr,\infty}) \ln \frac{p_\ast}{p_0} \, dx \\
&\quad = \varepsilon \int_\Omega (n_{tr} - n_{tr,\infty}) \bigg(\ln \frac{1 - n_{tr,\infty}}{n_{tr,\infty}} - \psi_\infty \bigg) dx \\
&\quad = -\varepsilon \int_\Omega \bigg( (n_{tr} - n_{tr,\infty}) \psi_\infty + \int_{n_{tr,\infty}}^{n_{tr}} \ln \frac{n_{tr,\infty}}{1 - n_{tr,\infty}} \, ds \bigg) dx.
\end{align*}
The relative entropy now reads
\begin{align*}
&E(n,p,n_{tr}, \psi) - E(n_\infty, p_\infty, n_{tr,\infty}, \psi_\infty) \\
&\quad = \int_\Omega \bigg( n \ln \frac{n}{n_\infty} - (n - n_\infty) + p \ln \frac{p}{p_\infty} - (p - p_\infty) \\
&\qquad + \frac{\lambda}{2} |\nabla \psi|^2 - \frac{\lambda}{2} |\nabla \psi_\infty|^2 - \big(n - n_\infty - p + p_\infty + \varepsilon (n_{tr} - n_{tr,\infty})\big) \psi_\infty \\
&\qquad + \varepsilon \int_{n_{tr,\infty}}^{n_{tr}} \bigg( \ln \frac{s}{1-s} - \ln \frac{n_{tr,\infty}}{1-n_{tr,\infty}} \bigg) ds \bigg) dx.
\end{align*}
Poisson's equation $n - n_\infty - p + p_\infty + \varepsilon (n_{tr} - n_{tr,\infty}) = -\lambda \Delta (\psi - \psi_\infty)$ and an integration by parts entail
\begin{align*}
&E(n,p,n_{tr}, \psi) - E(n_\infty, p_\infty, n_{tr,\infty}, \psi_\infty) \\
&\quad = \int_\Omega \bigg( n \ln \frac{n}{n_\infty} - (n - n_\infty) + p \ln \frac{p}{p_\infty} - (p - p_\infty) + \frac{\lambda}{2} |\nabla \psi|^2- \frac{\lambda}{2} |\nabla \psi_\infty|^2 \\
&\qquad - \lambda \nabla (\psi - \psi_\infty) \cdot \nabla \psi_\infty + \varepsilon \int_{n_{tr,\infty}}^{n_{tr}} \bigg( \ln \frac{s}{1-s} - \ln \frac{n_{tr,\infty}}{1-n_{tr,\infty}} \bigg) ds \bigg) dx.
\end{align*}
The claim now obviously follows from collecting the terms involving $\psi$ and $\psi_\infty$.
\end{proof}

Following ideas in \cite{GG96,FK18} and \cite{FK20}, we are able to bound the relative entropy essentially in terms of the squared $L^2$ distance between $(n, p, \sqrt{n_{tr}})$ and $(n_\infty, p_\infty, \sqrt{n_{tr,\infty}})$.

\begin{proposition}
\label{propeg}
There exists an explicit constant $c_1(K_\infty
) > 0$ satisfying
\begin{multline*}
E(n, p, n_{tr}, \psi) - E(n_\infty, p_\infty, n_{tr,\infty}, \psi_\infty) \\ 
\leq c_1 \int_\Omega \bigg( \frac{(n-n_\infty)^2}{n_\infty} + \frac{(p-p_\infty)^2}{p_\infty} + \varepsilon \big(\sqrt{n_{tr}} - \sqrt{n_{tr,\infty}} \big)^2 \bigg) \, dx
\end{multline*}
for all $\varepsilon > 0$ and all non-negative $n, p, n_{tr} \in L^2(\Omega)$, $n_{tr} \leq 1$ where $\psi \in H^1(\Omega)$ is the unique solution of \eqref{eqpsi} with $f = n - p + \varepsilon n_{tr} - D$ and $\ol \psi = 0$.
\end{proposition}
\begin{proof}
Applying the elementary inequality $\ln x \leq x - 1$ for $x > 0$, we derive
\[
n \ln \frac{n}{n_\infty} - (n-n_\infty) \leq n \Big( \frac{n}{n_\infty} - 1 \Big) - n + n_\infty = \frac{(n-n_\infty)^2}{n_\infty}
\]
and an analogous estimate involving $p$ and $p_\infty$. Integration by parts with homogeneous Neumann conditions for $\psi$ and $\psi_\infty$ as well as $-\lambda \Delta (\psi - \psi_\infty) = (n - n_\infty) - (p - p_\infty) + \varepsilon (n_{tr} - n_{tr,\infty})$ yield
\begin{align*}
&\lambda \int_{\Omega} \left| \nabla (\psi-\psi_\infty)\right|^2 dx \\ 
&\quad = \int_\Omega \big( (n - n_\infty) - (p - p_\infty) + \varepsilon (n_{tr} - n_{tr,\infty}) \big) (\psi - \psi_\infty) \, dx \\
&\quad \leq \frac{1}{2} \bigg( \frac1\delta \| (n - n_\infty) - (p - p_\infty) + \varepsilon (n_{tr} - n_{tr,\infty}) \|^2 + \delta \| \psi - \psi_\infty \|^2 \bigg) \\
&\quad \leq \frac{3\LL}{2\lambda} \Big( \| n - n_\infty \|^2 + \| p - p_\infty \|^2 + \varepsilon \| n_{tr} - n_{tr,\infty} \|^2 \Big) + \frac{\lambda}{2} \| \nabla (\psi - \psi_\infty) \|^2.
\end{align*}
Here and below, we abbreviate $\| \cdot \| \coleq \| \cdot \|_{L^2(\Omega)}$ and denote by $L(\Omega) > 0$ a constant such that Poincar\'e's estimate $\| f \|^2 \leq \LL \| \nabla f \|^2$ holds true for all $f \in H^1(\Omega)$ subject to $\ol f = 0$.
The estimate in the last line is then a result of $\ol{\psi - \psi_\infty} = 0$ and Poincar\'e's inequality together with the choice $\delta \coleq \lambda/\LL$, whereas the previous bound follows from H\"older's inequality and Young's inequality with some constant $\delta > 0$. We thus find
\begin{multline*}
\frac{\lambda}{2} \int_{\Omega} \left| \nabla (\psi-\psi_\infty) \right|^2 dx \\ 
\leq \frac{3\LL \max \{ M_\infty, 4\}}{2\lambda} \int_\Omega \bigg( \frac{(n-n_\infty)^2}{n_\infty} + \frac{(p-p_\infty)^2}{p_\infty} + \varepsilon \big(\sqrt{n_{tr}} - \sqrt{n_{tr,\infty}} \big)^2 \bigg) \, dx,
\end{multline*}
where we employed the bounds from \eqref{eqeqbounds} and $\big(\sqrt{n_{tr}} + \sqrt{n_{tr,\infty}} \big)^2 \leq 4$.

The last term within the relative entropy including $n_{tr}$ can be controlled as in \cite{FK20}. For convenience, we briefly recall the main arguments. First, there exists for all $x \in \Omega$ some mean value 
\[
\theta(x) \in (\min\{n_{tr}(x), n_{tr,\infty}(x)\}, \max\{n_{tr}(x), n_{tr,\infty}(x)\})
\]
such that
\begin{equation}
\label{eqlnmeanvalue}
\int_{n_{tr,\infty}(x)}^{n_{tr}(x)} \ln \frac{s}{1-s} ds = (n_{tr}(x) - n_{tr,\infty}(x)) \ln \frac{\theta(x)}{1-\theta(x)}.
\end{equation}
To enhance readability, we shall suppress the $x$-dependence of $n_{tr}$ and $n_{tr,\infty}$ subsequently. We further use the bound $n_{tr,\infty} \in (\mu_\infty, 1-\mu_\infty)$ from \eqref{eqeqbounds} and observe that
\[
\left| \int_{n_{tr,\infty}}^{n_{tr}} \ln \frac{s}{1-s} ds \right| \leq \int_{0}^1 \left| \ln \frac{s}{1-s} \right| ds = 2\ln 2
\]
for all $x \in \Omega$. In combination with \eqref{eqlnmeanvalue}, this estimate entails
\[
\left| \ln \frac{\theta(x)}{1-\theta(x)} \right| \big| n_{tr} - n_{tr,\infty} \big| \leq 2 \ln 2.
\]
By an elementary argumentation, one can now conclude that $\theta(x) \in (\xi, 1 - \xi)$ where $\xi \in \big( 0,\frac12 \big)$ only depends on $\mu_\infty$. Therefore, we obtain
\begin{align*}
&\varepsilon \int_\Omega \int_{n_{tr,\infty}}^{n_{tr}} \left( \ln \frac{s}{1-s} - \ln \frac{n_{tr,\infty}}{1-n_{tr,\infty}} \right) ds \, dx \\
&\qquad = \varepsilon \int_\Omega \left( \ln \frac{\theta(x)}{1-\theta(x)} - \ln \frac{n_{tr,\infty}}{1-n_{tr,\infty}} \right) (n_{tr} - n_{tr,\infty}) \, dx \\
&\qquad = \varepsilon \int_\Omega \frac{1}{\sigma(x)(1-\sigma(x))} (\theta(x) - n_{tr,\infty}) (n_{tr} - n_{tr,\infty}) \, dx
\end{align*}
with some $\sigma(x) \in (\min\{\theta(x), n_{tr,\infty}(x)\}, \max\{\theta(x), n_{tr,\infty}(x)\}) \subset [\xi, 1-\xi]$ employing the mean-value theorem and taking into account that
\[
\frac{d}{ds} \ln \frac{s}{1-s} = \frac{1}{s(1-s)}.
\]
As $(\sigma(x)(1-\sigma(x)))^{-1}$ is uniformly bounded in $\Omega$ in terms of $\xi(\mu_\infty)$, there exists some $c > 0$ only depending on $\mu_\infty$ such that 
\begin{align*}
&\varepsilon \int_\Omega \int_{n_{tr,\infty}}^{n_{tr}} \left( \ln \frac{s}{1-s} - \ln \frac{n_{tr,\infty}}{1-n_{tr,\infty}} \right) ds \, dx \\
&\quad \leq c \varepsilon \int_\Omega |\theta(x) - n_{tr,\infty}| |n_{tr} - n_{tr,\infty}| \, dx \\
&\quad \leq 4 c \varepsilon \int_\Omega \big(\sqrt{n_{tr}} - \sqrt{n_{tr,\infty}} \big)^2 \, dx
\end{align*}
where the last line results from estimating $\big(\sqrt{n_{tr}} + \sqrt{n_{tr,\infty}} \big)^2 \leq 4$. This proves the claim.
\end{proof}

The subsequent lemma contains rather non-intuitive estimates for bilinear terms like $(n - n_\infty)(p - p_\infty)$. These expressions will appear in the proof of Proposition \ref{propgd} below. Admissible functions are typically assumed to belong to the set 
\begin{align}
\label{eqsetn}
\mathcal N \coleq \big\{ (n, p, n_{tr}) \in L^2_+(\Omega)^3 \; : \; n, p \leq M, \ n_{tr} \leq 1 \text{\ a.e.\ in\ } \Omega \big\}.
\end{align}

\begin{lemma}
\label{lemmareactionterms}
The following estimates hold true for all $(n, p, n_{tr}) \in \mathcal N$ with explicit constants $\Gamma_1(M
) > 0$ and $\Gamma_2
> 0$:
\begin{align*}
(n - n_\infty) (p - p_\infty) &\leq \Gamma_1 \bigg(\!- R_n \ln \frac{n (1 - n_{tr})}{n_0 \mu_n n_{tr}} - R_p \ln \frac{p n_{tr}}{p_0 \mu_p (1 - n_{tr})} \, \bigg), \\
(n - n_\infty) (-n_{tr} + n_{tr,\infty}) &\leq \Gamma_2 \bigg(\!- R_n \ln \frac{n (1 - n_{tr})}{n_0 \mu_n n_{tr}} + \big( \sqrt{n_{tr}} - \sqrt{n_{tr,\infty}} \big)^2 \, \bigg), \\
(p - p_\infty) (n_{tr} - n_{tr,\infty}) \\
&\mqquad\mqquad\mquad \leq \Gamma_2 \bigg(\!- R_p \ln \frac{p n_{tr}}{p_0 \mu_p (1 - n_{tr})} + \big( \sqrt{1 - n_{tr}} - \sqrt{1 - n_{tr,\infty}} \big)^2 \, \bigg).
\end{align*}
\end{lemma}
\begin{proof}
As in \cite{GG96} and \cite{FK18}, we first recall the elementary inequalities $(a - a_0)(b - b_0) \leq (\sqrt{ab} - \sqrt{a_0 b_0})^2$ for all $a,a_0,b,b_0 \geq 0$ 
and $4 (\sqrt x - \sqrt y)^2 \leq (x - y) \ln \frac{x}{y}$ for all $x \geq 0$ and $y > 0$.

Concerning the first inequality we write 
\[
(n - n_\infty)(p - p_\infty) \leq \big( \sqrt{np} - \sqrt{n_\infty p_\infty} \big)^2 = n_\infty p_\infty \bigg(\sqrt{\frac{np}{n_0 \mu_n p_0 \mu_p}} - 1 \bigg)^2
\]
and distinguish the two cases $n_{tr} > \frac12$ and $n_{tr} \leq \frac12$. In the case $n_{tr} > \frac12$, we infer
\begin{align*}
(n - n_\infty) (p - p_\infty) &\leq n_0 p_0 \mu_n \mu_p \bigg( \sqrt{ \frac{n}{n_0 \mu_n n_{tr}} } \bigg( \sqrt{ \frac{p}{p_0 \mu_p} n_{tr} } - \sqrt{1 - n_{tr}} \bigg) \\
&\quad + \sqrt{ \frac{1}{n_{tr}} } \bigg( \sqrt{ \frac{n}{n_0 \mu_n} (1 - n_{tr}) } - \sqrt{n_{tr}} \bigg) \bigg)^2 \\
&\leq \Gamma_1(M
) \bigg( \Big( \frac{n}{n_0 \mu_n} \, (1 - n_{tr}) - n_{tr} \Big) \ln \frac{n (1 - n_{tr})}{n_0 \mu_n n_{tr}} \\
&\quad + \Big( \frac{p}{p_0 \mu_p} \, n_{tr} - (1 - n_{tr}) \Big) \ln \frac{p n_{tr}}{p_0 \mu_p (1-n_{tr})} \bigg)
\end{align*}
employing the $L^\infty$ bound on $n$. Using analogous arguments we derive the same result also in the case $n_{tr} \leq \frac12$. The second inequality arises from
\begin{align*}
(n - n_\infty)(-n_{tr} + n_{tr,\infty}) &\leq \Big( \sqrt{n(1 - n_{tr})} - \sqrt{n_\infty (1 - n_{tr,\infty})} \Big)^2 \\
&= n_0 \mu_n \bigg( \sqrt{ \frac{n}{n_0 \mu_n} (1 - n_{tr}) } - \sqrt{n_{tr,\infty}} \bigg)^2
\end{align*}
where we used the relation $n_\infty (1 - n_{tr,\infty}) = n_0 \mu_n n_{tr,\infty}$ from Proposition \ref{propequilibrium}. The claim is now a consequence of
\begin{multline*}
n_0 \mu_n \bigg( \bigg( \sqrt{ \frac{n}{n_0 \mu_n} (1 - n_{tr}) } - \sqrt{n_{tr}} \bigg) + \big( \sqrt{n_{tr}} - \sqrt{n_{tr,\infty}} \big) \bigg)^2 \\
\leq \Gamma_2
\bigg( \Big( \frac{n (1 - n_{tr})}{n_0 \mu_n} - n_{tr} \Big) \ln \frac{n (1 - n_{tr})}{n_0 \mu_n n_{tr}} + \big( \sqrt{n_{tr}} - \sqrt{n_{tr,\infty}} \big)^2 \bigg).
\end{multline*}
Similarly, one can also verify the third inequality stated above.
\end{proof}

The next result establishes an upper bound for the $L^2$ distance between $(n,p)$ and $(n_\infty,p_\infty)$ basically in terms of the entropy production $P$ and $\| \sqrt{n_{tr}} - \sqrt{n_{tr,\infty}} \|_{L^2(\Omega)}^2$. Similar arguments already appeared in \cite{GG96} and \cite{FK18}.

\begin{proposition}
\label{propgd}
There exists an explicit constant $c_2(M, K_\infty
) > 0$ satisfying
\begin{multline*}
\int_\Omega \bigg( \frac{(n-n_\infty)^2}{n_\infty} + \frac{(p-p_\infty)^2}{p_\infty} \bigg) \, dx \leq c_2 P(n, p, n_{tr}, \psi) \\
+ c_2 \, \varepsilon \int_\Omega \bigg( \big( \sqrt{n_{tr}} - \sqrt{n_{tr,\infty}} \big)^2 + \big( \sqrt{1 - n_{tr}} - \sqrt{1 - n_{tr,\infty}} \big)^2 \bigg) dx 
\end{multline*}
for all $\varepsilon > 0$ and all $(n, p, n_{tr}) \in \mathcal N$ where additionally $n, p \in H^1(\Omega)$ and $\psi \in H^1(\Omega)$ is the unique solution of \eqref{eqpsi} with $f = n - p + \varepsilon n_{tr} - D$ and $\ol \psi = 0$.
\end{proposition}
\begin{proof}
We start by defining intermediate equilibria $N \coleq n_\ast e^{-\psi - V_n}$ and $P \coleq p_\ast e^{\psi - V_p}$ which fulfil $J_n(N, \psi) = J_p(P, \psi) = 0$. Due to $J_n(n, \psi) = N \nabla (\frac{n}{N})$ and $J_p(p, \psi) = P \nabla (\frac{p}{P})$, we derive the following lower bounds for the flux terms involving $J_n$ and $J_p$:
\begin{align*}
\frac{|J_n|^2}{n_\infty n} &= \frac{N^2}{n_\infty n} \bigg| \nabla \Big( \frac{n}{N} \Big) \bigg|^2 = \frac{N^2}{n_\infty n} \bigg| \frac{n_\infty}{N} \nabla \Big( \frac{n}{n_\infty} \Big) + \frac{n}{n_\infty} \nabla \Big( \frac{n_\infty}{N} \Big) \bigg|^2 \\
&= \frac{N^2}{n_\infty n} \bigg| e^{\psi - \psi_\infty} \nabla \Big( \frac{n}{n_\infty} \Big) + \frac{n}{n_\infty} e^{\psi - \psi_\infty} \nabla (\psi - \psi_\infty) \bigg|^2 \\
&\geq 2 \frac{N^2}{n_\infty^2} e^{2(\psi - \psi_\infty)} \nabla \Big( \frac{n}{n_\infty} \Big) \cdot \nabla (\psi - \psi_\infty) = 2 \nabla (\psi - \psi_\infty) \cdot \nabla \Big( \frac{n - n_\infty}{n_\infty} \Big).
\end{align*}
In the same way, we obtain $\frac{|J_p|^2}{p_\infty p} \geq -2 \nabla (\psi - \psi_\infty) \cdot \nabla ( \frac{p - p_\infty}{p_\infty} )$ and, therefore, 
\begin{multline*}
\frac{\lambda}{2} \int_{\Omega} \bigg( \frac{|J_n|^2}{n_\infty n} + \frac{|J_p|^2}{p_\infty p} \bigg) dx \geq \lambda \int_\Omega \nabla (\psi - \psi_\infty) \cdot \nabla \bigg( \frac{n - n_\infty}{n_\infty} - \frac{p - p_\infty}{p_\infty} \bigg) dx \\
= \int_\Omega \Big( (n - n_\infty) - (p - p_\infty) + \varepsilon (n_{tr} - n_{tr,\infty}) \Big) \bigg( \frac{n - n_\infty}{n_\infty} - \frac{p - p_\infty}{p_\infty} \bigg) dx
\end{multline*}
via integration by parts and Poisson's equation. Rearranging this inequality now yields
\begin{align*}
&\int_\Omega \bigg( \frac{(n-n_\infty)^2}{n_\infty} + \frac{(p-p_\infty)^2}{p_\infty} \bigg) dx \\
&\quad \leq \frac{\lambda}{2} \int_{\Omega} \bigg( \frac{|J_n|^2}{n_\infty n} + \frac{|J_p|^2}{p_\infty p} \bigg) dx + \int_{\Omega} \bigg( \Big( \frac{1}{n_\infty} + \frac{1}{p_\infty} \Big) (n - n_\infty)(p - p_\infty) \\
&\qquad + \varepsilon (-n_{tr} + n_{tr,\infty}) \frac{n - n_\infty}{n_\infty} + \varepsilon (n_{tr} - n_{tr,\infty}) \frac{p - p_\infty}{p_\infty} \bigg) dx.
\end{align*}
Together with the bounds from \eqref{eqeqbounds} and Lemma \ref{lemmareactionterms}, we arrive at the desired result.
\end{proof}

We are now in a position to prove the EEP inequality from Theorem \ref{theoremeepinequality}, where the main task is to provide an appropriate bound on $(\sqrt{n_{tr}} - \sqrt{n_{tr,\infty}})^2$.

\begin{proof}[Proof of Theorem \ref{theoremeepinequality}] 
\noindent
\textbf{Step 1.} Due to \eqref{eqntrrelations} we easily calculate 
\begin{align*}
\sqrt{n_{tr}} &- \sqrt{\frac{n}{n_0 \mu_n} (1 - n_{tr})} = \sqrt{n_{tr}} - \sqrt{n_{tr,\infty}} - \sqrt{\frac{n}{n_0 \mu_n} (1 - n_{tr})} \\
&\qquad + \sqrt{\frac{n_\infty}{n_0 \mu_n} (1 - n_{tr})} - \sqrt{\frac{n_\infty}{n_0 \mu_n} (1 - n_{tr})} + \sqrt{\frac{n_\infty}{n_0 \mu_n} (1 - n_{tr,\infty})} \\
&= \sqrt{n_{tr}} - \sqrt{n_{tr,\infty}} - \sqrt{\frac{1 - n_{tr}}{n_0 \mu_n}} \Big( \sqrt{n} - \sqrt{n_\infty} \Big) \\
&\qquad + \sqrt{\frac{n_\infty}{n_0 \mu_n}} \Big( \sqrt{1 - n_{tr,\infty}} - \sqrt{1 - n_{tr}} \Big).
\end{align*}
Observing that $\sqrt{n_{tr}} - \sqrt{n_{tr,\infty}}$ and $\sqrt{1 - n_{tr,\infty}} - \sqrt{1 - n_{tr}}$ have the same sign and using the inequality $4 (\sqrt x - \sqrt y)^2 \leq (x - y) \ln \frac{x}{y}$ for $x \geq 0$ and $y > 0$, we reformulate the previous identity to find
\begin{align*}
\big( \sqrt{n_{tr}} - \sqrt{n_{tr,\infty}} \big)^2 &\leq \bigg( \sqrt{n_{tr}} - \sqrt{n_{tr,\infty}} + \sqrt{\frac{n_\infty}{n_0 \mu_n}} \Big( \sqrt{1 - n_{tr,\infty}} - \sqrt{1 - n_{tr}} \Big) \bigg)^2 \\
&\leq 2 \bigg( \sqrt{n_{tr}} - \sqrt{\frac{n}{n_0 \mu_n} (1 - n_{tr})} \bigg)^2 + \frac{2}{n_0 \mu_n} \big( \sqrt{n} - \sqrt{n_\infty} \big)^2 \\
&\leq \frac12 \Big( \frac{n (1 - n_{tr})}{n_0 \mu_n} - n_{tr} \Big) \ln \frac{n (1 - n_{tr})}{n_0 \mu_n n_{tr}} + \frac{2}{n_0 \mu_n} \frac{(n - n_\infty)^2}{n_\infty}.
\end{align*}
Along the same lines, we also deduce 
\begin{multline*}
\big( \sqrt{1 - n_{tr}} - \sqrt{1 - n_{tr,\infty}} \big)^2 \\
\leq \frac12 \Big( \frac{p n_{tr}}{p_0 \mu_p} - (1 - n_{tr}) \Big) \ln \frac{p n_{tr}}{p_0 \mu_p (1 - n_{tr})} + \frac{2}{p_0 \mu_p} \frac{(p - p_\infty)^2}{p_\infty}.
\end{multline*}

\smallskip \noindent
\textbf{Step 2.} We can now improve the claim of Proposition \ref{propeg} in the sense that there exists a constant $c_1(K_\infty
) > 0$ satisfying
\begin{multline*}
E(n, p, n_{tr}, \psi) - E(n_\infty, p_\infty, n_{tr,\infty}, \psi_\infty) \\
\leq c_1 P(n, p, n_{tr}, \psi) + c_1 \int_\Omega \bigg( \frac{(n-n_\infty)^2}{n_\infty} + \frac{(p-p_\infty)^2}{p_\infty} \bigg) \, dx
\end{multline*}
for all $(n, p, n_{tr}) \in \mathcal N$ where $\psi \in H^1(\Omega)$ is the unique solution of \eqref{eqpsi} with $f = n - p + \varepsilon n_{tr} - D$ and $\ol \psi = 0$ for any $\varepsilon > 0$. Furthermore, we notice that Proposition \ref{propgd} now gives rise to a constant $c_2(M, K_\infty
) > 0$ such that
\begin{multline*}
\int_\Omega \bigg( \frac{(n-n_\infty)^2}{n_\infty} + \frac{(p-p_\infty)^2}{p_\infty} \bigg) \, dx \\
\leq c_2 P(n, p, n_{tr}, \psi) + c_2 \, \varepsilon \int_\Omega \bigg( \frac{(n - n_\infty)^2}{n_\infty} + \frac{(p - p_\infty)^2}{p_\infty} \bigg) dx 
\end{multline*}
holds true for all $\varepsilon > 0$ and all $(n, p, n_{tr}) \in \mathcal N$ with $n, p \in H^1(\Omega)$ and $\psi \in H^1(\Omega)$ being the unique solution of \eqref{eqpsi} with $f = n - p + \varepsilon n_{tr} - D$ and $\ol \psi = 0$.

\smallskip \noindent
\textbf{Step 3.} If we restrict $\varepsilon$ to the interval $\big( 0, \tfrac{1}{2c_2} \big)$, we finally arrive at
\[
E(n, p, n_{tr}, \psi) - E(n_\infty, p_\infty, n_{tr,\infty}, \psi_\infty) \leq (c_1 + 2 c_1 c_2) P(n, p, n_{tr}, \psi)
\]
employing the notation from Step 2.
\end{proof}

\section{Proof of the exponential convergence}
\label{secconv}
As soon as the weak entropy production law \eqref{eqweakeplaw} is settled, the exponential decay of the relative entropy arises from a Gronwall argument as carried out in \cite{FK20} (see also \cite{Wil65,Bee75}).

\begin{proof}[Proof of Theorem \ref{theoremexpconvergence}]
The weak entropy production law \eqref{eqweakeplaw} readily follows from \eqref{eqentropy} and \eqref{eqproduction} for $0 < t_0 \leq t_1 < \infty$ utilizing the regularity and bounds on $n$, $p$, $n_{tr}$, and $\psi$ from Proposition \ref{propglobalsolution}.
The statement of the theorem is then a consequence of Theorem \ref{theoremeepinequality} applied to the global solution $(n, p, n_{tr}, \psi)$. Note, however, that the weak entropy production law \eqref{eqweakeplaw} only allows to derive 
\begin{multline*}
E(n,p,n_{tr},\psi)(t) - E(n_\infty, p_\infty, n_{tr,\infty}, \psi_\infty) \\
\leq (E(n,p,n_{tr},\psi)(t_0) - E(n_\infty, p_\infty, n_{tr,\infty}, \psi_\infty)) e^{-C_{EEP}^{-1} (t-t_0)}
\end{multline*}
for all $t_0 \in (0, t]$. The assertion in \eqref{eqexpconventropy} is then a consequence of the fact that the entropy $E(n,p,n_{tr},\psi)(t_0)$ continuously extends to $t_0 \rightarrow 0$ since $n,p \in C([0,T), L^2(\Omega))$, $n_{tr} \in C([0,T), L^\infty(\Omega))$, and $\psi \in C([0,T), H^2(\Omega))$ for all $T>0$ by Proposition \ref{propglobalsolution}.
\end{proof}

Exponential convergence in $L^\infty$ and $H^2$ for $(n,p,n_{tr})$ and $\psi$, respectively, is a consequence of standard regularity techniques, which have been partially employed already in \cite{FK20}. As a prerequisite, we formulate a Csisz\'ar--Kullback--Pinsker-type inequality, which we believe to be well-known. But as we were not able to find a precise reference, we provide a proof in the subsequent lemma. 

\begin{lemma}[A Csisz\'ar--Kullback--Pinsker-type inequality]
\label{lemmackp}
Let $f, g : \Omega \rightarrow \mathbb R$ be non-negative and measurable functions and $g$ be strictly positive. Then, 
\[
\int_\Omega \bigg( f \ln \frac{f}{g} - f + g \bigg) dx \geq \frac{3}{2 \ol f + 4 \ol g} \| f - g \|_{L^1(\Omega)}^2.
\]
\end{lemma}
\begin{proof}
Going back to an idea of Pinsker, we first prove the elementary inequality $h(u) \coleq (2u+4)(u \ln u - u + 1) - 3(u-1)^2 \geq 0$ for scalar $u \geq 0$. The claim follows from the identities $h(1) = h'(1) = h''(1) = 0$ and the sign of $h'''(u) = \tfrac{4}{u} - \tfrac{4}{u^2}$ for $u > 1$ and $u < 1$, respectively. As a consequence, we obtain 
\begin{align*}
\| f - g \|_{L^1(\Omega)} &= \int_\Omega \bigg| \frac{f}{g} - 1 \bigg| g \, dx \leq \int_\Omega \frac{g}{\sqrt{3}} \sqrt{\frac{2f}{g} + 4} \sqrt{\frac{f}{g} \ln \frac{f}{g} - \frac{f}{g} + 1} \, dx \\ 
&\leq \frac{1}{\sqrt{3}} \sqrt{\int_\Omega (2f + 4g) \, dx} \sqrt{\int_\Omega \Big( f \ln \frac{f}{g} - f + g \Big) \, dx},
\end{align*}
where we employed H\"older's inequality in the last step.
\end{proof}

\begin{proof}[Proof of Corollary \ref{corconvlinfty}]
An immediate consequence of the exponential decay of the relative entropy as stated in \eqref{eqexpconventropy} is the exponential convergence to the equilibrium of $n(t)$ and $p(t)$ in $L^1(\Omega)$, and of $n_{tr}(t)$ in $L^2(\Omega)$. To see this, we first recall the explicit representation of the relative entropy from Lemma \ref{lemmarelativeentropy}. 
Lemma \ref{lemmackp} allows us to control
\[
\int_\Omega \left( n \ln \frac{n}{n_\infty} - (n - n_\infty) \right) dx \geq \frac{3}{2 \ol n + 4 \ol{n_\infty}} \| n - n_\infty \|_{L^1(\Omega)}^2 \geq c \| n - n_\infty \|_{L^1(\Omega)}^2
\]
and analogously $\| p - p_\infty \|_{L^1(\Omega)}^2$ in terms of a (rough) constant $c(M, K_\infty) > 0$. Next, we notice that
\[
\frac{d}{ds} \ln \left( \frac{s}{1-s} \right) = \frac{1}{s(1-s)} \geq 4
\]
is valid for all $s \in (0,1)$, which enables us to estimate 
\begin{multline*}
\varepsilon \int_\Omega \int_{n_{tr,\infty}}^{n_{tr}} \left( \ln \left( \frac{s}{1-s} \right) - \ln \left( \frac{n_{tr,\infty}}{1 - n_{tr,\infty}} \right) \right) ds \, dx \\
= \varepsilon \int_\Omega \int_{n_{tr,\infty}}^{n_{tr}} \frac{1}{\sigma(s)(1 - \sigma(s))} (s - n_{tr,\infty}) \, ds \, dx 
\geq 2 \varepsilon  \| n_{tr} - n_{tr,\infty} \|_{L^2(\Omega)}^2
\end{multline*}
where $\sigma(s)$ serves as an intermediate point between $n_{tr,\infty}$ and $s$.

As a preparation for the exponential convergence of $n$ and $p$ in $L^\infty(\Omega)$, we adapt an argument from \cite{FK20} to establish a polynomially growing $W^{1,q}(\Omega)$ bound on $n$ for $q \geq 4$. (In fact, we shall only require such a bound for $q = 6$ below.) The same technique is also applicable to $p$. As in the proof of Proposition \ref{propglobalsolution}, we 
set w.l.o.g.\ $\tau_n = \tau_p = 1$ leading to 
\begin{align*}
\partial_t n = \Delta n + \nabla n \cdot \nabla (\psi + V_n) + n \Delta (\psi + V_n) + n_{tr} - \frac{n}{n_0 \mu_n} (1-n_{tr}).
\end{align*}
Using $-|\nabla n|^{q-2} \Delta n$ as a test function and recalling $\hat n \cdot \nabla n = 0$ on $\partial \Omega$ entails
\begin{multline*}
\frac1{q(q-1)} \frac{d}{dt} \int_\Omega |\nabla n|^q \, dx = \frac1{q-1} \int_\Omega |\nabla n|^{q-2} \nabla n \cdot \nabla \partial_t n \, dx = - \int_\Omega |\nabla n|^{q-2} \Delta n \, \partial_t n \, dx \\
= - \int_\Omega |\nabla n|^{q-2} |\Delta n|^2 \, dx - \int_\Omega |\nabla n|^{q-2} \Delta n \nabla n \cdot \nabla (\psi + V_n) \, dx \\
- \int_\Omega |\nabla n|^{q-2} \Delta n \, n \Delta (\psi + V_n) \, dx - \int_\Omega |\nabla n|^{q-2} \Delta n \Big(n_{tr} - \frac{n}{n_0 \mu_n} (1 - n_{tr}) \Big) \, dx.
\end{multline*}
By estimating the third line with Young's inequality via 
\[
\Big| \Delta n \, n \Delta (\psi + V_n) + \Delta n \Big(n_{tr} - \frac{n}{n_0 \mu_n} (1 - n_{tr}) \Big) \Big| \leq \frac12 |\Delta n|^2 + \frac12 C_2^2
\]
with a constant $C_2(M) > 0$, and by observing that 
\begin{multline*}
\Big| \int_\Omega |\nabla n|^{q-2} \Delta n \nabla n \cdot \nabla (\psi + V_n) \, dx \Big| = \Big| \int_\Omega \frac{1}{q} \nabla \big( |\nabla n|^q \big) \cdot \nabla (\psi + V_n) \, dx \Big| \\ 
\leq \frac{1}{q} \int_\Omega |\nabla n|^q \big| \Delta (\psi + V_n) \big| \, dx \leq \frac1q \int_\Omega |\nabla n|^q \, C_1 \, dx 
\end{multline*}
with another constant $C_1(M)>0$, we calculate 
\begin{multline*}
\frac1{q(q-1)} \frac{d}{dt} \int_\Omega |\nabla n|^q \, dx 
\leq - \int_\Omega |\nabla n|^{q-2} |\Delta n|^2 \, dx \\
+ \frac1q \int_\Omega |\nabla n|^q \, C_1 \, dx + \int_\Omega |\nabla n|^{q-2} \Big( \frac12 |\Delta n|^2 + \frac12 C_2^2 \Big) \, dx.
\end{multline*}
We rewrite the first term in the second line by another integration by parts and Young's inequality, which leads us to 
\begin{align*}
\frac1q \int_\Omega |\nabla n|^q \, dx &= \frac1q \int_\Omega |\nabla n|^{q-2} \nabla n \cdot \nabla n \, dx 
= - \frac{q-1}{q} \int_\Omega |\nabla n|^{q-2} \Delta n \, n \, dx \\
&\leq \frac{1}{2C_1} \int_\Omega |\nabla n|^{q-2} |\Delta n|^2 \, dx + \frac{C_1M^2}{2} \int_\Omega |\nabla n|^{q-2} \, dx.
\end{align*}
The previous estimates now guarantee that
\[
\frac{d}{dt} \int_\Omega |\nabla n|^q \, dx \leq C_3 \int_\Omega |\nabla n|^{q-2} \, dx,
\]
with a constant $C_3(M, q) > 0$. Choosing $t_0 > 0$ and $t \geq t_0$ arbitrarily and utilizing $|\Omega| = 1$, one has
\[
\| \nabla n(t) \|_{L^q(\Omega)}^q \leq \| \nabla n(t_0) \|_{L^q(\Omega)}^q + C_3 \int_{t_0}^t \| \nabla n(s) \|_{L^q(\Omega)}^{q-2} \, ds.
\]
The polynomial growth of $\| \nabla n \|_{L^q(\Omega)}$ is then obtained by an elementary Gronwall lemma (see e.g. \cite{Bee75}). In detail, 
\begin{align}
\label{eqboundgradn}
\| \nabla n(t) \|_{L^q(\Omega)} \leq \Big( \| \nabla n(t_0) \|_{L^q(\Omega)}^2 + C_3 (t-t_0) \Big)^{\frac12}.
\end{align}

Exponential convergence to the equilibrium for $n$ in $L^q(\Omega)$, $1 < q < \infty$, is easily deduced from the exponential convergence of $n$ in $L^1(\Omega)$ as settled above and the $L^\infty(\Omega)$ bounds on $n$ and $n_\infty$ in \eqref{equpperboundnp} and \eqref{eqeqbounds}, respectively, by writing
\[
\| n - n_\infty \|_{L^q(\Omega)}^q \leq \| n - n_\infty \|_{L^\infty(\Omega)}^{q-1} \| n - n_\infty \|_{L^1(\Omega)} \leq C e^{-c t}
\]
where the constants $c(M, K_\infty, q), C(M, K_\infty, q) > 0$ are independent of $\varepsilon$ for $\varepsilon \in (0, \varepsilon_0]$. For $q = 2$ and together with the same bound on $p$ and the estimate 
\[
\| \psi - \psi_\infty \|_{H^2(\Omega)} \leq C \big( \| n - n_\infty \|_{L^2(\Omega)} + \| p - p_\infty \|_{L^2(\Omega)} + \varepsilon \| n_{tr} - n_{tr,\infty} \|_{L^2(\Omega)} \big),
\]
this directly implies the exponential convergence of $\psi$ in $H^2(\Omega)$.
The Gagliardo--Nirenberg--Moser interpolation inequality now allows us to infer exponential convergence of $n$ and $p$ in $L^\infty(\Omega)$. In fact, the bound on $\| \nabla n \|_{L^{6}(\Omega)}$ in \eqref{eqboundgradn} entails
\begin{equation}
\label{eqconvlinfty}
\|n-n_\infty\|_{L^{\infty}(\Omega)} \le C \| n-n_\infty \|_{W^{1,6}(\Omega)}^{\frac12} \| n-n_\infty \|_{L^{6}(\Omega)}^{\frac12} \leq C e^{-c t},
\end{equation}
with constants $c(M, K_\infty), C(M, K_\infty) > 0$. 

The exponential convergence of $n_{tr}$ in $L^\infty(\Omega)$ can be verified essentially along the same lines as in \cite{FK20}. We, therefore, omit some technical details and set w.l.o.g.\ $\tau_n = \tau_p = 1$. By defining $u \coleq n_{tr} - n_{tr,\infty}$, one derives the following pointwise relation by inserting $\pm n_{tr,\infty}$ several times and by applying the identities from \eqref{eqntrrelations}:
\begin{align*}
\varepsilon \, \partial_t u &= R_p - R_n = \bigg( 1 - n_{tr} - \frac{p}{p_0 \mu_p} n_{tr} \bigg) - \bigg( n_{tr} - \frac{n}{n_0 \mu_n} (1 - n_{tr}) \bigg) \\
&= - u \left(2 + \frac{p_\infty}{p_0 \mu_p} + \frac{n_\infty}{n_0 \mu_n}\right) - \frac{n_{tr}}{p_0 \mu_p} \left(p-p_\infty\right) + \frac{(1 - n_{tr})}{n_0 \mu_n} \left(n-n_\infty\right).
\end{align*}
By recalling $0 \leq n_{tr} \leq 1$, $\mu_n = e^{-V_n}$, and $\mu_p = e^{-V_p}$ with $V_n, V_p \in L^\infty(\Omega)$, and by employing \eqref{eqconvlinfty}, we find 
\begin{align*}
\frac{d}{dt} \|u(t, \cdot)\|_{L^\infty(\Omega)} 
\leq - \frac{2}{\varepsilon} \|u(t, \cdot)\|_{L^\infty(\Omega)} + \frac{C}{\varepsilon} e^{-c t}
\end{align*}
where $c, C > 0$ depend on $M$ and $K_\infty$ but not on $\varepsilon$ for $\varepsilon \in (0, \varepsilon_0]$. If we choose $c > 0$ sufficiently small satisfying $\varepsilon_0 c \leq 1$, we arrive at
\begin{multline*}
\|n_{tr}(t, \cdot) - n_{tr,\infty} \|_{L^\infty(\Omega)} \le e^{-2 t/\varepsilon} + \frac{C}{\varepsilon} \int_0^t e^{ -2 (t-s) /\varepsilon - c s}\,ds \\
\le e^{-2 t/\varepsilon} + e^{-2 t /\varepsilon} \frac{C}{2 - \varepsilon c} \left( e^{ (2/\varepsilon - c)t}  - 1\right)
\le e^{-2 t/\varepsilon_0} + C e^{ - c t}.
\end{multline*}
Finally, estimate \eqref{eqexpconvlinfty} is proven.
\end{proof}

\section{Conclusion and Outlook}
We have derived a so-called entropy--entropy production (EEP) inequality for a recombination--drift--diffusion--Poisson system modelling the dynamics of electrons and holes on separate energy levels in semiconductor materials. We have then employed this EEP inequality to establish exponential convergence to the equilibrium for the densities of the involved charge carriers and the self-consistent electrostatic potential. However, several simplifying hypotheses have been imposed (cf. Assumption \ref{assump}), which allowed for a transparent presentation of the main ideas of the proof but which prevent one from directly applying the results to real-world semiconductor devices. 

We conclude the article with a couple of comments on possible future research.
\begin{itemize}
	\item Instead of considering trapped states allowing for a limited number of electrons, one can---at least from a mathematical perspective---also consider trapped states attracting holes. This could be achieved by replacing $n_{tr}$ by the density of \emph{trapped holes} $p_{tr}$, and by appropriately reformulating model \eqref{eqsystem}. As the structure of the resulting system remains essentially unchanged, we believe that our results also transfer to this situation. 
	\item Concerning the presence of multiple trap levels or even a continuous distribution of trap levels within the bandgap of the semiconductor as in \cite{GMS07}, we stress that this requires a different definition of the entropy functional involving the density $n_{tr}^\eta$ of occupied trapped states with energy $\eta \in [E_\mathrm{min}, E_\mathrm{max}]$. Assuming all constants to be independent of $\eta$ (which also enforces the equilibria $n_{tr,\infty}^\eta$ to coincide), one can redefine the entropy functional as  
	\begin{align}
	\label{eqentropynew}
	E(n, p, \{n_{tr}^\eta\}_\eta, \psi) &\coleq \int_{\Omega} \bigg( n \ln \frac{n}{n_0 \mu_n} - (n-n_0\mu_n) + p \ln \frac{p}{p_0 \mu_p} - (p-p_0\mu_p) \\ &\qquad \quad + \frac{\lambda}{2} \big| \nabla \psi \big|^2 + \varepsilon \int_{E_\mathrm{min}}^{E_\mathrm{max}} \int_{1/2}^{n_{tr}^\eta} \ln \left( \frac{s}{1-s} \right) ds \, d\eta \bigg) dx. \nonumber
	\end{align}
	In this context, our strategy leads to the same results by following the line of arguments with some minor adaptions. But as soon as $n_0$ and $p_0$ depend on $\eta$, it already seems to be infeasible to derive an expression for the entropy production $P$ similar to \eqref{eqproduction} (at least if the integral over $\eta$ is placed in front of the right hand side of \eqref{eqentropynew}). Further studies in this direction shall be carried out in a subsequent project.
	\item In order to generalise our framework to the setting of semiconductor devices subject to non-trivial boundary conditions and space-dependent material parameters, the following questions have to be addressed: How to prove the existence of global solutions (observing that heterostructures are not covered by \cite{GMS07})? How to incorporate boundary conditions into the entropy functional preserving the non-negativity of $E$ and $P$ along with $P = -\tfrac{d}{dt} E$? Moreover, the fluxes $J_n$, $J_p$ and the reactions $R_n$, $R_p$ are now typically non-zero in the equilibrium. Is it, therefore, reasonable to expect relative fluxes and relative reactions to appear in the entropy production? To our knowledge, questions of equilibration of such models are largely open. 
\end{itemize}

\bibliographystyle{unsrt}
\bibliography{TrappedStatesSelfCons}

\end{document}